\documentclass{amsart}
\usepackage{mathrsfs,amssymb,multirow,setspace,color}
\usepackage[a4paper, total={7in, 9in}]{geometry}
\theoremstyle{plain}
\usepackage{graphicx}
\newtheorem{theorem}{Theorem}[section]
\newtheorem{lemma}[theorem]{Lemma}

\newtheorem{corollary}[theorem]{Corollary}

\theoremstyle{definition}

\newtheorem{example}[theorem]{Example}

\theoremstyle{remark}
\newtheorem{remark}[theorem]{Remark}

\input xy
\xyoption{all}
\def\a{\mathcal{P}_E(G)}
\def\c{\overline{\mathcal{P}_E(G)}}
\def\d{\overline{\mathcal{P}_E(G^*)}}
\def\m{\mathcal{M}(G)}
\def\g{\mathcal{G}_{\m}}
\begin{document}
\title[The complement of enhanced power graph of a finite group]{The complement of enhanced power graph of a finite group}



\author[Parveen, Jitender Kumar]{Parveen, $\text{Jitender Kumar}^{^*}$ }
\email{p.parveenkumar144@gmail.com,jitenderarora09@gmail.com}

\begin{abstract}
The enhanced power graph $\a$ of a finite group $G$ is the simple undirected graph whose vertex set is $G$ and two distinct vertices $x, y$ are adjacent if $x, y \in \langle z \rangle$ for some $z \in G$. In this article, we give an affirmative answer of the question posed by Cameron \cite[Question 20]{a.cameron2021graphs} which states that: Is it true that the complement of the enhanced power graph $\c$ of a non-cyclic group $G$ has only one connected component apart from isolated vertices$?$ We classify all finite groups $G$ such that the graph $\c$ is bipartite. We show that the graph $\c$ is weakly perfect. Further, we study the subgraph $\d$ of $\c$ induced by all the non-isolated vertices of $\c$. We classify all finite groups $G$ such that the graph is $\d$ is unicyclic and pentacyclic. We prove the non-existence of finite groups $G$ such that the graph $\d$ is bicyclic, tricyclic or tetracyclic. Finally, we characterize all finite groups $G$ such that the graph $\d$ is outerplanar, planar, projective-planar and toroidal, respectively.
\end{abstract}

\subjclass[2020]{05C25}

\keywords{Enhanced power graph, maximal cyclic subgroup, genus of a graph. \\ *  Corresponding author}

\maketitle
\section{Introduction}
The study of graphs related to various algebraic structures becomes important, because graphs of this type have valuable applications and are related to automata theory (see \cite{a.kelarev2009cayley,kelarev2004labelled} and the books \cite{kelarev2003graph,kelarev2002ring}). Certain graphs, viz. power graphs, commuting graphs, Cayley graphs etc., associated to groups have been studied by various researchers, see \cite{a.chakrabarty2009undirected,a.kelarev2002undirected,a.segev2001commuting}.  Segev \cite{a.segev1999finite,a.segev2001commuting}, Segev and Seitz \cite{a.segev2002anisotropic} used combinatorial parameters of certain commuting graphs to establish long standing conjectures in the theory of division algebras. A variant of commuting graphs on groups has played an important role in classification of finite simple groups, see \cite{b.aschbacher2000finite}. Hayat \emph{et al.} \cite{a.hayat2019novel} used commuting graphs associated with groups to establish some NSSD (non-singular with a singular deck) molecular graph.

In order to measure how much the power graph is close to the commuting graph of a group $G$, Aalipour \emph{et al.} \cite{a.Cameron2016} introduced a new graph called \emph{enhanced power graph.} The enhanced power graph of a group $G$ is the simple undirected graph whose vertex set is $G$ and two distinct vertices $x, y$ are adjacent if $x, y \in \langle z \rangle$ for some $z \in G$. Indeed, the enhanced power graph contains the power graph and is a spanning subgraph of the commuting graph. Aalipour \emph{et al.} \cite{a.Cameron2016} characterized the finite group $G$, for which equality holds for either two of the three graphs viz. power graph, enhanced power graph and commuting graph of $G$. Further, the enhanced power graphs have received the considerable attention by various researchers. Bera \emph{et al.} \cite{a.Bera2017} characterized the abelian groups and the non abelian $p$-groups having dominatable enhanced power graphs. Dupont \emph{et al.} \cite{a.dupont2017rainbow} determined the rainbow connection number of enhanced power graph of a finite group $G$. Later, Dupont \emph{et al.} \cite{a.dupont2017enhanced} studied the graph theoretic properties of enhanced quotient graph of a finite group $G$.  A complete description of finite groups with enhanced power graphs admitting a perfect code have been studied in \cite{a.ma2017perfect}. Ma \emph{et al.} \cite{a.ma2020metric} investigated the metric dimension of an enhanced power graph of finite groups. Hamzeh et al. \cite{a.hamzeh2017automorphism} derived the automorphism groups of enhanced power graphs of finite groups.  Zahirovi$\acute{c}$ \emph{et al.} \cite{a.zahirovic2020study} proved that two finite abelian groups are isomorphic if their enhanced power graphs are isomorphic. Also, they supplied a characterization of finite nilpotent groups whose enhanced power graphs are perfect.  Recently, Panda \emph{et al.} \cite{a.panda2021enhanced} studied the graph-theoretic properties, viz. minimum degree, independence number, matching number, strong metric dimension and perfectness, of enhanced power graphs over finite abelian groups. Moreover, the enhanced power graphs associated to non-abelian groups such as semidihedral, dihedral, dicyclic, $U_{6n}$, $V_{8n}$ etc., have been studied in \cite{a.dalal2021enhanced, a.panda2021enhanced}. Bera \emph{et al.} \cite{a.bera2021connectivity} gave an upper bound for the vertex connectivity of enhanced power graph of any finite abelian group. Moreover, they classified the finite abelian groups whose proper enhanced power graphs are connected. The connectivity of the complement of the enhanced power graph has been studied in \cite{a.ma2021note}. Abdollahi \emph{et al.} \cite{a.Abdollahi2007} investigated some graph theoretic properties (such as diameter, regularity) of the non-cyclic graph $\Gamma (G)$ of a group $G$. Note that the graph $\Gamma (G)$ is same as the complement of the enhanced power graph of $G$. Moreover, they proved that the clique number of $\Gamma (G)$ is finite if and only if $\Gamma (G)$ has no infinite clique.

Cameron \cite{a.cameron2021graphs} posed a question that: Is it true that the complement of the enhanced power graph of a finite non-cyclic group has just one connected components apart from isolated vertices? Then a natural question arises as to investigate graph theoretic properties of the connected graph obtained by removing isolated vertices. In order to investigate this question, in this paper, we consider the complement of the enhanced power graph of a finite group $G$. The paper is arranged as follows. In Section $2$, we provide the necessary background material and fix our notations used throughout the paper. In section $3$, we have proved that the complement of the enhanced power graph of a finite group has just one  connected components apart from isolated vertices. Moreover, we obtain the girth and the chromatic number  of $\c$. Further, we classify all finite groups such that the subgraph $\d$, obtained by deleting isolated vertices, of $\a$  is dominatable, Eulerian, unicyclic and pentacyclic, respectively.  In Section $4$, we classify all finite groups $G$ such that the graph $\d$ is outerplanar, planar, projective-planar and toroidal, respectively.
\section{Preliminaries}
In this section, first we recall the graph theoretic notions from  \cite{b.westgraph}. A \emph{graph} $\Gamma$ is a pair  $\Gamma = (V, E)$, where $V(\Gamma)$ and $E(\Gamma)$ are the set of vertices and edges of $\Gamma$, respectively. Two distinct vertices $u_1, u_2$ are $\mathit{adjacent}$, denoted by $u_1 \sim u_2$, if there is an edge between $u_1$ and $u_2$. Otherwise, we write it as $u_1 \nsim u_2$. Let $\Gamma$ be a graph. A \emph{subgraph}  $\Gamma'$ of $\Gamma$ is the graph such that $V(\Gamma') \subseteq V(\Gamma)$ and $E(\Gamma') \subseteq E(\Gamma)$. For $X \subseteq V(\Gamma)$, the subgraph of $\Gamma$ induced by the set $X$ is the graph with vertex set $X$ and its two distinct vertices are adjacent if and only if they are adjacent in $\Gamma$. The \emph{complement} $\overline{\Gamma}$ of $\Gamma$ is a graph with same vertex set as $\Gamma$ and distinct vertices $u, v$ are adjacent in $\overline{\Gamma}$ if they are not adjacent in $\Gamma$. A graph $\Gamma$ is said to be \emph{complete} if any two distinct vertices are adjacent. We denote $K_n$ by the complete graph of $n$ vertices. A graph $\Gamma$ is said to be $k$-partite if the vertex set of $\Gamma$ can be partitioned into $k$ subsets, such that no two vertices in the same subset of the partition are adjacent. A \emph{complete k-partite} graph, denoted by $K_{n_1,n_2,\ldots ,n_k}$, is a $k$-partite graph having its parts sizes $n_1,n_2,\ldots ,n_k$ such that every vertex in each part is adjacent to all the vertices of all other parts of $K_{n_1,n_2,\ldots ,n_k}$. A vertex $v$ of $\Gamma$ is said to be a \emph{dominating vertex} if $v$ is adjacent to all the other vertices of $\Gamma$.  The \emph{degree} $deg(v)$ of a vertex $v$ in a graph $\Gamma$, is the number of edges incident to $v$.  A \emph{walk} $\lambda$ in $\Gamma$ from the vertex $u$ to the vertex $w$ is a sequence of vertices $u = v_1, v_2, \ldots , v_m = w (m > 1)$ such that $v_i \sim v_{i+1}$ for every $i \in \{1, 2, \ldots , m-1\}$. If no edge is repeated in $\lambda$, then it is called a \emph{trail} in $\Gamma$. A trail whose initial and end vertices are identical is called a \emph{closed
trail}. A walk is said to be a \emph{path} if no vertex is repeated. A graph $\Gamma$ is \emph{connected}  if each pair of vertices has a path in $\Gamma$. Otherwise, $\Gamma$ is \emph{disconnected}. A graph $\Gamma$ is \emph{Eulerian} if $\Gamma$ is both connected and has a closed trail (walk with no repeated edge) containing all the edges of a graph. The \emph{distance} between $u, v \in V(\Gamma)$, denoted by $d(u, v)$,  is the number of edges in a shortest path connecting them. A \emph{clique} of a graph $\Gamma$ is a complete subgraph of $\Gamma$ and the number of vertices in a clique of maximum size is called the \emph{clique number} of $\Gamma$ and it is denoted by $\omega (\Gamma)$. The \emph{chromatic number} $\chi(\Gamma)$ of a graph $\Gamma$ is the smallest positive integer $k$ such that the vertices of $\Gamma$ can be colored in $k$ colors so that no two adjacent vertices share the same color. An \emph{independent set} of a graph $\Gamma$ is a subset of $V (\Gamma)$ such that no two vertices in the subset are adjacent in $\Gamma$. The \emph{independence number} of $\Gamma$ is the maximum size of an independent set, it is denoted by $\alpha(\Gamma)$. The graph $\Gamma$ is \emph{weakly perfect} if $\omega(\Gamma)=\chi(\Gamma)$.
  The \emph{diameter} of $\Gamma$ is the maximum distance between the pair of vertices in $\Gamma$. A graph $\Gamma$ is \emph{planar} if it can be drawn on a plane without edge crossing. A planar graph is said to be \emph{outerplanar} if it can be drawn in the plane such that its all vertices lie on the outer face.  A graph is said to be \emph{embeddable} on a topological surface if it can be drawn on the surface without edge crossing. The \emph{orientable genus} or \emph{genus} of a graph $\Gamma$, denoted by $\gamma (\Gamma)$, is the smallest non-negative integer $n$ such that $\Gamma$ can be embedded on the sphere with $n$ handles. Note that the graphs having genus $0$ are planar graphs and graphs having genus $1$ are \emph{toroidal} graphs. Let $\mathbb{N}_k$ denotes the non-orientable surface formed by connected sum of $k$ projective planes, that is, $\mathbb{N}_k$ is a non-orientable surface with $k$ crosscap. The crosscap of a graph $\Gamma$, denoted by $\overline{\gamma}(\Gamma)$, is the minimum non-negative integer $k$ such that $\Gamma$ can be embedded in $\mathbb{N}_k$. For instance, a graph $\Gamma$ is planar if $\overline{\gamma}(\Gamma)=0$ and the $\Gamma$ is \emph{projective-planar} if $\overline{\gamma}(\Gamma)=1$.
  The following results are used in the subsequent sections.
    \begin{theorem}{\cite{b.westgraph}}{\label{bipartitecondition}}
A  graph $\Gamma$ is bipartite if and only if it has no odd cycle.
\end{theorem}
  \begin{theorem}{\cite{b.westgraph}}{\label{outerplanarcondition}}
A graph $\Gamma$ is outerplanar if and only if it does not contain a subdivision of $K_4$ or $K_{2,3}$.
\end{theorem}
\begin{theorem}{\cite{b.westgraph}}{\label{planarcondition1}}
A graph $\Gamma$ is planar if and only if it does not contain a subdivision of $K_5$ or $K_{3,3}$.
\end{theorem}
\begin{theorem}{\cite{b.White1973}}{\label{planarcondition}} The genus and cross-cap of the complete graphs $K_n$ and $K_{m, n}$ are given below:
\begin{itemize}
\item[(i)] $\gamma(K_n)= \left\lceil{\frac{(n-3)(n-4)}{12}}\right\rceil $, $n\geq 3$.\\
\item[(ii)]$\gamma(K_{m,n})=\left\lceil \frac{(m-2)(n-2)}{4}\right\rceil $, $m,n\geq 2$.\\
\item[(iii)] $\overline{\gamma}(K_n)= \left\lceil{\frac{(n-3)(n-4)}{6}}\right\rceil $, $n\geq 3$, $n\neq 7$; $\overline{\gamma} (K_n)=3$ if $n=7$.\\
\item[(iv)] $\overline{\gamma}(K_{m,n})=\left\lceil \frac{(m-2)(n-2)}{2}\right\rceil $, $m,n\geq 2$.\\
\end{itemize}
\end{theorem}
Let $G$ be a group. The order of an element $x$ in $G$ is denoted by $o(x)$. For a positive integer $n$, $\phi(n)$ denotes the Euler's totient function of $n$. For $n \geq 3$, the \emph{dihedral group} $D_{2n}$ is a group of order $2n$ is defined in terms of generators and
relations as $D_{2n} = \langle x, y  :  x^{n} = y^2 = e,  xy = yx^{-1} \rangle$. For $n \geq 2$, the \emph{dicyclic group} $Q_{4n}$ is a group of order $4n$ is defined in terms of generators and
relations as $Q_{4n} = \langle a, b  :  a^{2n}  = e, a^n= b^2, ab = ba^{-1} \rangle.$ A cyclic subgroup of a group $G$ is called a \emph{maximal cyclic subgroup} if it is not properly contained in any cyclic subgroup of $G$ other than itself. If $G$ is a cyclic group, then $G$ is  the only maximal cyclic subgroup of $G$. We denote $\m$ by the set of all maximal cyclic subgroups of $G$. Also, we write $\g=\{x\in G : \langle x \rangle \in \m \}$.
\begin{remark}{\label{remark maximal}}
Let $G$ be a finite group. Then $G= \bigcup\limits_{M\in \m} M$ and the generators of a maximal cyclic subgroup does not belong to any other maximal cyclic subgroup of $G$. Consequently, if $|M_i|$ is a prime number then $M_i \cap M_j =\{e\}$ for distinct $M_i,  M_j \in \m$.
\end{remark}
The detail of the maximal cyclic subgroups of non-isomorphic groups of order up to $15$ is given in the Table \ref{table}. 
\begin{table}[ht]
\centering
\scalebox{.98}{
\begin{tabular}{ |l|c|c|c|l| } \hline
$O(G)$ & No. of Groups & Type of Groups & $|\m|$ & $|M_i|$ \\
\hline
\multirow{1}{4em}{1} & 1 & $\mathbb{Z}_1$ & 1 &  $|M_1|=1$\\ 
\hline
\multirow{1}{4em}{2} & 1 & $\mathbb{Z}_2$ & 1 & $|M_1|=2$\\ 
\hline
\multirow{1}{4em}{3} & 1 & $\mathbb{Z}_3$ & 1 & $|M_1|=3$\\
\hline
\multirow{2}{4em}{4} & 2 & $\mathbb{Z}_4$ & 1 & $|M_1|=4$\\
& & $\mathbb{Z}_2\times \mathbb{Z}_2$ & 3 & $|M_1|=|M_2|=|M_3|=2$\\
\hline
\multirow{1}{4em}{5} & 1 & $\mathbb{Z}_5$ & 1 & $|M_1|=5$\\
\hline
\multirow{2}{4em}{6} & 2 & $\mathbb{Z}_6$ & 1 & $|M_1|=6$\\
& & $S_3$ & 4 & $|M_1|=3, |M_2|=|M_3|=|M_4|=2$\\
\hline
\multirow{1}{4em}{7} & 1 & $\mathbb{Z}_7$ & 1 & $|M_1|=7$\\
\hline
\multirow{5}{4em}{8} & 5 & $\mathbb{Z}_8$ & 1 & $|M_1|=8$\\
& & $\mathbb{Z}_2\times \mathbb{Z}_4$ & 4 & $|M_1|= |M_2|=4$, $|M_3|=|M_4|=2$\\
& & $\mathbb{Z}_2\times \mathbb{Z}_2\times \mathbb{Z}_2$ & 7 & $|M_i|= 2$ for $i\in \{1,2,3,4,5,6,7\}$\\
& & $D_8$ & 5 & $|M_1|= 4, \ |M_i|=2$ for $i\in \{2,3,4,5\}$\\
& & $Q_8$ & 3 & $|M_1|= |M_2|=|M_3|=4$\\
\hline
\multirow{2}{4em}{9} & 2 & $\mathbb{Z}_9$ & 1 & $|M_1|=9$\\
& & $\mathbb{Z}_3\times \mathbb{Z}_3$ & 4 & $|M_1|=|M_2|=|M_3|=|M_4|=3$\\
\hline
\multirow{2}{4em}{10} & 2 & $\mathbb{Z}_{10}$ & 1 & $|M_1|=10$\\
& & $D_{10}$ & 6 & $|M_1|=5, \ |M_i|=2$ for $i\in \{2,3,4,5,6\}$\\
\hline
\multirow{1}{4em}{11} & 1 & $\mathbb{Z}_{11}$ & 1 & $|M_1|=11$\\
\hline
\multirow{5}{4em}{12} & 5 & $\mathbb{Z}_{12}$ & 1 & $|M_1|=12$\\
& & $\mathbb{Z}_2\times \mathbb{Z}_6$ & 3 & $|M_1|= |M_2|=|M_3|=6$\\
& & $A_4$ & 7 & $|M_i|= 3$, $|M_j|=2$ for $i\in \{1,2,3,4\}$,  $j\in \{5,6,7\}$\\
& & $D_{12}$ & 7 & $|M_1|= 6, \ |M_i|=2$ for $i\in \{2,3,4,5,6,7\}$\\
& & $Q_6$ & 4 & $|M_1|=6, \ |M_2|=|M_3|=|M_4|=4$\\
\hline
\multirow{1}{4em}{13} & 1 & $\mathbb{Z}_{13}$ & 1 & $|M_1|=13$\\
\hline
\multirow{2}{4em}{14} & 2 & $\mathbb{Z}_{14}$ & 1 & $|M_1|=14$\\
& & $D_{14}$ & 8 & $|M_1|=7, \ |M_i|=2$ for $i\in \{2,3,4,5,6,7,8\}$\\
\hline
\multirow{1}{4em}{15} & 1 & $\mathbb{Z}_{15}$ & 1 & $|M_1|=15$\\
\hline
\end{tabular}
}
\vspace{.3cm}
\hspace{-1cm} \caption{The maximal cyclic subgroups of non-isomorphic groups of order upto 15.}
    \label{table}
\end{table}

The following results are used in the subsequent sections.
\begin{theorem}{\rm \cite[Theorem 2.4]{a.Bera2017}}{\label{complete}}
The enhanced power graph $\a$   of the group $G$  is complete if and only if $G$ is cyclic.
\end{theorem}
\begin{lemma}{\label{2maximal}}
If $G$ is a finite group then $| \m |\neq 2$.
\end{lemma}
\begin{proof}
On contrary, assume that the group $G$ has two maximal cyclic subgroups $M_1$ and $M_2$. Then every element of $G$ belongs to at least one of the maximal cyclic subgroup of $G$ and $e\in M_1\cap M_2$. It follows that 
$|M_1|+|M_2|\geq o(G)+1.$
Since $M_1$ and $M_2$ are proper subgroups of a finite group $G$, by Lagrange's theorem, we have $$|M_1|\leq \frac{o(G)}{2} \ \text{and} \  |M_2|\leq \frac{o(G)}{2}.$$
Consequently, we get $o(G)+1 \leq |M_1|+|M_2| \leq o(G)$, which is not possible. Hence, $| \m |\neq 2$.
\end{proof}


\section{Graph invariants and the structure of $\c$} 
In this section, first we give an affirmative answer to the question posed by Cameron \cite[Question 20]{a.cameron2021graphs} (cf. Theorem \ref{1connected}). We classify all finite groups $G$ such that the graph $\c$ is bipartite. We obtain the girth of $\c$ and also we prove that the graph $\c$ is weakly perfect. Further, we study the subgraph $\d$ of $\c$ induced by all the non-isolated vertices in $\c$. Then we classify the groups $G$ such that the graph $\d$ is dominatable and Eulerian, respectively. Finally, we classify all finite groups $G$ such that the graph is $\d$ is unicyclic and pentacyclic. We also proved that the graph $\d$ can not be bicyclic, tricyclic and tetracyclic.

\begin{lemma}{\label{intersectionlemma}}
Let $G$ be a finite group. Then $x$ is an isolated vertex of $\c$ if and only if $x$ lies in every maximal cyclic subgroup of $G$.
\end{lemma}
\begin{proof}
First suppose that $x$ lies in every maximal cyclic subgroup of $G$. This implies that  $x$ is adjacent to every element of $G$ in $\a$ and so $x$ is an isolated vertex in $\c$. Conversely, let $x$ be an isolated vertex in $\c$. Then $x$ is a dominatable vertex in $\a$. Consequently, $x$ belongs to every maximal cyclic subgroup of $G$.
\end{proof}
\begin{theorem}{\label{1connected}}
Let $G$ be a finite non-cyclic group. Then $\c$ has just one connected component, apart from isolated vertices.
\end{theorem}
\begin{proof}
Let $M$ be a maximal cyclic subgroup of $G$ and let $x\in \widetilde{M}$, where $\widetilde{M}$ is the set of generators of $M$. First note that $x$ is adjacent to every element of $G\setminus M$ in $\c$. If possible, assume that $x\nsim y$ in $\c$, for some $y\in G\setminus M$. Then $x$ is adjacent to $y$ in $\a$. Therefore, $x,y \in \langle z \rangle$ for some $z\in G$. This contradicts the maximality of  $M$. Thus, $(G\setminus M)\cup \widetilde{M}$ is a connected component of $\c$. Now for $z\in M\setminus \widetilde{M}$, if $z\in \cap M_i$ for every $i$, then by Lemma \ref{intersectionlemma}, $z$ is an isolated vertex. Now, if $z\notin M_j$ for some $M_j\in \m$, then $z$ is adjacent to every element of $\widetilde{M}_j \subseteq (G\setminus M)$ in $\c$. Thus, the result holds.
\end{proof}
\begin{theorem}{\label{bipartitegirth}}
Let $G$ be a finite group. Then the following hold:
\begin{itemize}
    \item[(i)] the graph $\c$ is bipartite if and only if $G$ is cyclic.
   \item[(ii)] the girth of $\c$ is either $3$ or $\infty$.
   \item[(iii)]  the graph $\c$ is weakly perfect.
\end{itemize}
\end{theorem}

\begin{proof}
(i) Let $\c$ be a bipartite graph such that $|\m |\geq 3$. Suppose $M_1= \langle x \rangle ,\ M_2=\langle y \rangle$ and $M_3= \langle z \rangle$. Then we get a cycle $x\backsim y\backsim z\backsim x$ of odd length in $\c$; a contradiction (cf. Theorem \ref{bipartitecondition}).
Consequently, $|\m|\leq 2$. By Lemma \ref{2maximal}, we get $|\m|=1$. Thus, $G$ is a cyclic group. 
Conversely, assume that $G$ is a cyclic group of order $n$. By Theorem \ref{complete}, $\c$ is a null graph. Hence, $\c$ is a bipartite graph.\\
(ii) If $G$ is a cyclic group then $\c$ is a null graph and hence the girth of $\c$ is $\infty$. We may now suppose that $G$ is a non-cyclic group. By the proof of part (i), notice that $\c$ contains a cycle of length $3$. Thus, the girth of $\c$ is $3$.\\
(iii) By {\rm \cite[Theorem 3.3]{a.panda2021enhanced}}, $\alpha (\mathcal{P}_E(G))=|\m|$. For a graph $\Gamma$, we have $\alpha (\Gamma)=\omega (\overline{\Gamma})$. Thus, $\omega(\c)=|\m|$. Also, $\chi (\c)\geq \omega (\c)= |\m|$. Notice that if $G$ is a non-cyclic group then $\c$ is a $k-$partite graph, where $k=|\m|$. Consequently, $\chi (\c)\leq |\m|$. Thus the result holds. 
\end{proof}

Note that the subgraph $\d$ of $\c$ induced by all the non-isolated vertices in $\c$ is connected. In what follows, we study various graph invariants and  embeddings of $\d$ on various surfaces. 
\begin{theorem}
The graph $\d$ is dominatable if and only if $G$ has a maximal cyclic subgroup of order $2$.
\end{theorem}
\begin{proof}
Let $\d$ be a dominatable graph. Then there exists a vertex $x$ of $\d$ such that $x$ is adjacent to every vertex of $\d$. Note that if $o(x)\geq 3$, then $x\nsim x^{-1}$ in $\d$; a contradiction. It follows that $o(x)=2$. Moreover, note that $\langle x \rangle \in \m$. Otherwise, the generator of a maximal cyclic subgroup containing $x$ is not adjacent to $x$ in $\d$, which is not possible. Conversely, let $M=\langle x \rangle$ be a maximal cyclic subgroup of order $2$. Then $x\sim y$ for every $y\in G\setminus M$. Consequently, $x$ is a dominating vertex of $\d$. Thus, the result holds.
\end{proof}
\begin{corollary}
The graph $\d$ is complete if and only if $G \cong \mathbb{Z}_2\times \mathbb{Z}_2\times \cdots \times \mathbb{Z}_2.$
\end{corollary}


\begin{theorem}{\label{Eulerian}}
For $x \in G$, let $\mathrm{M}_x$ be the union of all the maximal cyclic subgroups of $G$ containing $x$. Then the graph $\d$ is  Eulerian if and only if either $|G|$ is odd or $|\mathrm{M}_x|$ is even for every $x\in V(\d)$.
\end{theorem}
\begin{proof}
If $|G|$ is odd then the enhanced power graph $\a$ is Eulerian (cf. {\rm \cite[Theorem 2.5]{a.Bera2017}}). Thus, the degree of every vertex of $\a$ is even (see {\rm \cite[Theorem 1.2.26]{b.westgraph}}). Let $x\in V(\d)$ such that $deg(x)=m$ in $\a$ and  let $|G|=n$. Thus, in $\d$, we have $deg(x)=n-m-1$ which is an even number. Consequently, $\d$ is Eulerian. We may now suppose that $|\mathrm{M}_x |$ is even for every vertex of $\d$. Now let $y$ be an arbitrary vertex of $\d$. Note that in $\d$, $deg(y)=|G|-|\mathrm{M}_y |$ is even. Thus, $\d$ is Eulerian. Conversely, suppose that $\d$ is Eulerian. If $|G|$ is odd then there is nothing to prove. If $|G|$ is not odd and $|\mathrm{M}_x |$ is odd for some $x\in V(\d)$, then $deg(x)=|G|-|\mathrm{M}_x |$ is odd; a contradiction.  Thus, $|\mathrm{M}_x |$ must be even for every $x\in V(\d)$.
\end{proof}
If $G$ is a $2$-group then every non-trivial subgroup of $G$ is of even order. Consequently, $|\mathrm{M}_x |$ is even for every $x\in V(\d)$. Thus we have the following corollary of Theorem \ref{Eulerian}.

\begin{corollary}
If $G$ is a $2$-group then the graph $\d$ is Eulerian.
\end{corollary}
Now we give examples of even order groups $G$ such that the graph $\d$ is Eulerian.
\begin{example}
Consider $G$ to be the dihedral group $D_{2n}$. Notice that $G$ has a maximal cyclic subgroup $M=\langle x \rangle$ of order $n$ and every element of $G$ belongs to exactly one maximal cyclic subgroup of $G$. If $n$ is odd then $deg(x)=|G|-|M|$ which is an odd number. Hence, $\overline{\mathcal{P}_E(D_{2n}^*)}$ is not Eulerian. If $n$ is even then every maximal cyclic subgroup of $D_{2n}$ is of even order. Consequently, $|\mathrm{M}_x |$ is even for every $x\in V(\overline{\mathcal{P}_E(D_{2n}^*)})$. It follows that $\overline{\mathcal{P}_E(D_{2n}^*)}$ is Eulerian. Thus, $\overline{\mathcal{P}_E(D_{2n}^*)}$ is Eulerian if and only if $n$ is even.
\end{example}
\begin{example}
 Consider $G$ to be the dicyclic group $Q_{4n}$. Observe that the centre of $Q_{4n}$ is contained in every maximal cyclic subgroup of $G$ and $x\in Q_{4n}\setminus Z(Q_{4n})$ belongs to exactly one maximal cyclic subgroup of $Q_{4n}$. Also, notice that $Q_{4n}$ has $1$ maximal cyclic subgroup of order $2n$ and $n$ maximal cyclic subgroup of order $4$. Consequently, $|\mathrm{M}_x |$ is even for every $x\in V(\overline{\mathcal{P}_E(Q_{4n}^*)})$. Thus, $\overline{\mathcal{P}_E(Q_{4n}^*)}$ is Eulerian for all $n\geq 2$.
\end{example}

Suppose $\Gamma$ is a connected graph with $n$ vertices and $m$ edges. If $c(\Gamma) =m-n+1$ then $\Gamma$ is called $c-$cyclic. We call $\Gamma$ to be unicyclic, bicyclic, tricyclic, tetracyclic and petacyclic if $c(\Gamma)=1,2,3,4$ and $5$, respectively. Clearly, $\Gamma$ is a tree if and only if $c(\Gamma)=0.$ The following lemma is easy to prove.

\begin{lemma}{\label{c of gamma}}
Let $\Gamma '$ be a connected subgraph of a connected graph $\Gamma$. Then $c(\Gamma ') \leq c(\Gamma ).$
\end{lemma}
Now we classify all finite groups $G$ such that $c(\d)\in \{1,2,3,4,5\}$.

\begin{theorem}{\label{unicyclic}}
Let $G$ be a finite non-cyclic group. Then the following hold:
\begin{itemize}
    \item[(i)] The graph $\d$ is unicyclic if and only if $G$ is isomorphic to $\mathbb{Z}_2\times \mathbb{Z}_2.$
   \item[(ii)] The graph $\d$ is pentacyclic if and only if $G$ is isomorphic to $S_3$.
   \item[(iii)] The graph $\d$ cannot be a bicyclic, tricyclic, and tetracyclic.
\end{itemize}
\end{theorem}
\begin{proof}
In view of Lemma \ref{2maximal}, we prove the result through the following cases on the cardinality of the set $\m$.
\noindent\textbf{Case-1:} $|\m |=3$. Let $M_1, M_2$ and $M_3$ be the maximal cyclic subgroups of $G$ such that $|M_i|=m_i$ for $i\in \{1,2,3\}$. Without loss of generality, assume that $\phi(m_1) \geq \phi(m_2) \geq \phi(m_3)$. Now, we have the following subcases:

\textbf{Subcase-1.1:} $\phi(m_1)=\phi(m_2)=\phi(m_3)=1.$ It follows that $m_1=m_2=m_3=2$. The identity element belongs to every maximal cyclic subgroup of $G$. Hence, we get $o(G)=4$ and by Table \ref{table}, $G\cong \mathbb{Z}_2 \times \mathbb{Z}_2$. Consequently,  $\d \cong K_3$ and $c(K_3)=1$. Thus, $\d$ is unicyclic. 

\textbf{Subcase-1.2:} $\phi(m_1)=2$ and $\phi(m_2)=1=\phi(m_3)$.
Then $m_2=m_3=2$. Consequently, by Remark \ref{remark maximal}, $m_1+2=o(G)$ and by Lagrange's theorem, $m_1\leq \frac{o(G)}{2}$. It follows that $o(G)\leq 4.$ But $\phi(m_1)=2$ implies that $m_1\geq 3$, which is not possible.

\textbf{Subcase-1.3:} $\phi(m_1)=2=\phi(m_2)$ and $\phi(m_3)=1$. It follows that $m_1,m_2 \in \{3,4,6\}$ and $m_3=2$. Now, let us assume that $m_j=3$ and $m_k=t$, for distinct $j,k\in \{1,2\}$ and $t\in \{3,4,6\}$. Then by Remark \ref{remark maximal}, $o(G)\in \{6,7,9\}$. By Table \ref{table}, no such group exists. Now if $m_1=m_2=4$ and $m_3=2$, then $|M_1\cap M_2|\in \{1,2\}$. Consequently, by Remark \ref{remark maximal}, $o(G)\in \{7,8\}$. The existence of a subgroup of order $4$ implies that $o(G)=8$. By Table \ref{table}, no such group exists.\\
 For distinct $i,j\in \{1,2\}$, if $m_i=4, m_j=6$ and $m_3=2$, then $|M_1\cap M_2|\in \{1,2\}$. Consequently, by Remark \ref{remark maximal}, $o(G)\in \{9,10\}$.  But the existence of a subgroup of order $6$ follows that no such group exists. If $m_1=m_2=6$ and $m_3=2$ then $|M_1\cap M_2|\in \{1,2,3\}$. Consequently, by Remark \ref{remark maximal}, $o(G)\in \{10,11,12\}$. Since there exists a subgroup of order $6$, we get $o(G)=12$. By Table \ref{table}, no such group exists.
  
\textbf{Subcase-1.4:} $\phi(m_1)=\phi(m_2)=\phi(m_3)=2$. Let $\Gamma$ be the  subgraph of $\d$ induced by the set $\g$. Then note that $c(\Gamma)=7$. By Lemma \ref{c of gamma}, $c(\d)\geq 7$.

\textbf{Subcase-1.5:} $\phi(m_1)\geq 4$ and $\phi(m_2)=\phi(m_3)=1$. Then $m_2=m_3=2$. By Remark \ref{remark maximal}, $m_1+2=o(G)$ and by Lagrange's theorem, $m_1 \leq \frac{o(G)}{2}$. It follows that $o(G)\leq 4$, and so $m_1\leq 2$, which is not possible.

\textbf{Subcase-1.6:} $\phi(m_1)\geq 4, \  \phi(m_2)\geq 2 $ and $ \phi(m_3)\geq 1$. Let $a_1,a_2,a_3,a_4$ be generators of $M_1$ and $b_1,b_2$ be generators of $M_2$. Further, suppose that $M_3= \langle c_1 \rangle$. Let $\Gamma$ be the subgraph of $\d$ induced by the set $S=\{a_1,a_2,a_3 ,a_4, b_1,b_2,c_1\}$. Then $c(\Gamma)=8$ and by Lemma \ref{c of gamma}, we get $c(\d)\geq 8.$

   

\noindent\textbf{Case-2:} $|\m |=4$. Let $M_1, M_2, M_3$ and $M_4$ be the maximal cyclic subgroups of $G$ such that $|M_i|=m_i$ for $i\in \{1,2,3,4\}$. Without loss of generality, assume that $\phi(m_1) \geq \phi(m_2) \geq \phi(m_3) \geq \phi(m_4)$. Now, we have the following subcases:

 \textbf{Subcase-2.1:} $\phi(m_1)=\phi(m_2)=\phi(m_3)=\phi(m_4)=1.$ Then $m_1=m_2=m_3=m_4=2$. The identity element belongs to every maximal cyclic subgroup of $G$. Hence, we obtain $o(G)=5$. By Table \ref{table}, no such group exists.
 
 \textbf{Subcase-2.2:} $\phi(m_1)=2$ and $\phi(m_2)=\phi(m_3)=\phi(m_4)=1$. Then $m_1\in \{3,4,6\}$ and $m_2=m_3=m_4=2$. Consequently, by Remark \ref{remark maximal}, $o(G)\in \{6,7,9\}$. The existence of a subgroup of order $2$ gives $o(G)=6$. By Table \ref{table}, $G\cong S_3$. By Figure \ref{S3}, we get $c(\overline{\mathcal{P}_E(S_3^*)}=5$. Thus, $\overline{\mathcal{P}_E(S_3^*)}$ is pentacyclic.
  
  \textbf{Subcase-2.3:} $\phi(m_1)\geq 2,\  \phi(m_2)\geq 2, \ \phi(m_3)\geq 1$ and $\phi(m_4)\geq 1$. Let $a_1,a_2$ be generators of $M_1$ and $b_1,b_2$ be generators of $M_2$. Further, let $M_3= \langle c_1 \rangle$ and $M_4= \langle d_1 \rangle$. Let $\Gamma$ be the subgraph of $\d$ induced by the set $S=\{a_1 ,a_2, b_1,b_2,c_1,d_1\}$. Then $c(\Gamma)=8$ and by Lemma \ref{c of gamma}, we get $c(\d)\geq 8.$
  
  \textbf{Subcase-2.4:} $\phi(m_1)\geq 4$ and $\phi(m_i)\geq 1$ for $i\in \{2,3,4\}$. Let $a_1, a_2, a_3, a_4$ be generators of $M_1$, and let $M_2 = \langle b_1 \rangle, M_3 = \langle c_1 \rangle$ and $M_4 = \langle d_1 \rangle $. Let $\Gamma$ be the subgraph of $\d$ induced by the set $S=\{a_1 ,a_2, a_3, a_4 ,b_1,c_1,d_1\}$. Then $c(\Gamma)=9$. By Lemma \ref{c of gamma}, we get $c(\d)\geq 9.$
  
\noindent\textbf{Case-3:} $|\m |\geq 5$. Consider the set $S=\{g_1,g_2,\ldots ,g_5\}$, where $M_i=\langle g_i\rangle \in \m$ for $i\in \{1,2,\ldots, 5\}$. Note that the
 subgraph of $\d$ induced by the set $S$ is isomorphic to the complete graph $K_5$. Since $c(K_5)=6$,  by Lemma \ref{c of gamma}, we obtain $c(\d)\geq 6.$
 
 Thus, by all the cases, the result holds.
\end{proof}

\section{Embedding of $\d$ on surfaces}
In this section, we study the embedding of the graph $\d$ on various surfaces without edge crossing. We classify all finite groups $G$ such that the graph $\d$ is outerplanar, planar, projective-planar and toroidal, respectively. Moreover, we show that there does not exist a group $G$ such that the cross-cap of the graph $\d$ is two. 
\begin{theorem}
Let $G$ be a finite non-cyclic group. Then 
\begin{itemize}
    \item[(i)] the graph $\d$ is outerplanar if and only if $G$ is isomorphic to $\mathbb{Z}_2\times \mathbb{Z}_2$.
   \item[(ii)] the graph $\d$ is planar if and only if $G$ is isomorphic to one of the three groups: $\mathbb{Z}_2\times \mathbb{Z}_2$, $S_3$, $Q_8$.
   \item[(iii)] the graph $\d$ is projective-planar if and only if $G$ is isomorphic to either $D_8$ or $\mathbb{Z}_2\times \mathbb{Z}_4$.
    \item[(iv)] the graph $\d$ cannot have cross-cap $2$.
     \item[(v)] the graph $\d$ is toroidal if and only if $G$ is isomorphic to one of the following $5$ groups: 
     \[D_8, \mathbb{Z}_2\times \mathbb{Z}_4, \mathbb{Z}_3\times \mathbb{Z}_3, \mathbb{Z}_2\times \mathbb{Z}_6, \mathbb{Z}_2\times \mathbb{Z}_2\times \mathbb{Z}_2.\]
\end{itemize}
\end{theorem}
\begin{proof}
In view of Lemma \ref{2maximal}, we prove the result through the following cases on the cardinality of the set $\m$.
\noindent\textbf{Case-1: } $|\m |=3$. Let $M_1, M_2$ and $M_3$ be the maximal cyclic subgroups of $G$ such that $|M_i|=m_i$ for $i\in \{1,2,3\}$. Without loss of generality, assume that $\phi(m_1) \geq \phi(m_2) \geq \phi(m_3)$. Now, we have the following subcases:

 \textbf{Subcase-1.1:} $\phi(m_1)=\phi(m_2)=\phi(m_3)=1.$ By Subcase-$1.1$ of the Theorem \ref{unicyclic}, we get $G\cong \mathbb{Z}_2\times \mathbb{Z}_2$.
 
 \textbf{Subcase-1.2:} $\phi(m_1)=2$ and $\phi(m_2)=\phi(m_3)=1$. By Subcase-$1.2$ of the Theorem \ref{unicyclic}, no such group exists.
 
 \textbf{Subcase-1.3:} $\phi(m_1)=2=\phi(m_2)$ and $\phi(m_3)=1.$ By Subcase-$1.3$ of the Theorem \ref{unicyclic}, no such group exists.
 
 \textbf{Subcase-1.4:} $\phi(m_1)=\phi(m_2)=\phi(m_3)=2.$ Then $m_i\in \{3,4,6\}$ for each $i\in \{1,2,3\}$. For distinct $i,j,k \in \{1,2,3\}$, if $m_i=m_j=3$ and $m_k=t$, where $t\in \{3,4,6\}$, then by Remark \ref{remark maximal}, $o(G)\in \{7,8, 10\}$. But the existence of a subgroup of order $3$ follows that no such group exists. For distinct $i,j,k \in \{1,2,3\}$, if $m_i=3$ and $m_j=m_k=4$, then $|M_j\cap M_k| \in \{1,2\}$. Consequently, we get $o(G)\in \{8,9\}$ (see by Remark \ref{remark maximal}), which is not possible because $G$ has subgroups of order $3$ and $4$. For distinct $i,j,k \in \{1,2,3\}$, if $m_i=m_j=4$ and $m_k=6$, then notice that the cardinality of the  intersection of any two of these maximal cyclic subgroups is at most $2$. Consequently, by Remark \ref{remark maximal}, $o(G)\in \{10,11,12\}$. Since there exists a maximal cyclic subgroup of order $6$, we obtain $o(G)=12$. By Table \ref{table}, no such group exists. If $m_1=m_2=m_3=4$, then by Remark \ref{remark maximal}, $o(G)\in \{8,9,10\}$. The existence of a subgroup of order $4$ gives $o(G)=8$. By Table \ref{table}, we have $G\cong Q_8$.\\
For distinct $i,j,k \in \{1,2,3\}$, if $m_i=m_j=6$ and $m_k=3$,  then $|M_i\cap M_j|\in \{1,2,3\}$. Consequently, we get $o(G)\in \{11,12,13\}$ (see Remark \ref{remark maximal}), and therefore $o(G)=12$. By Table \ref{table}, no such group exists. For distinct $i,j,k \in \{1,2,3\}$, if $m_i=m_j=6$ and $m_k=4$,  then by Remark \ref{remark maximal}, $o(G)\in \{11,12,13,14\}$. The existence of a subgroup of order $4$ gives $o(G)=12$. By Table \ref{table}, no such group exists. If $m_1=m_2=m_3=6$ then by Remark \ref{remark maximal}, and the existence of order $6$ gives $o(G)=12$. By Table \ref{table}, $G\cong \mathbb{Z}_2\times \mathbb{Z}_6$. For distinct $i,j,k \in \{1,2,3\}$, if $m_i=3,m_j=4$ and $m_k=6$, then $lcm(3,4,6)=12$ divides $o(G)$. The identity element belongs to every maximal cyclic subgroup of a group $G$. Consequently, by Remark \ref{remark maximal}, $o(G)\le 11$; a contradiction.

 \textbf{Subcase-1.5:} $\phi(m_1)\geq 4$ and $\phi(m_2)=\phi(m_3)=1$. By Subcase-$1.5$ of the Theorem \ref{unicyclic}, no such group exists.
 
 \textbf{Subcase-1.6:} $\phi(m_1)=4$, $\phi(m_2)=2$ and $\phi(m_3)=1$. Then $m_1\in \{5,8,10,12\}, \  m_2\in \{3,4,6\}$ and $m_3=2$.  If $m_1=5$ and $m_2=t$, where $t\in \{3,4,6\}$, then by Remark \ref{remark maximal}, we obtain $o(G)\in \{8,9,11\}$. The existence of subgroup of order $5$ implies that no such group exists.
If $m_1=s,$ $m_2=t$, where $s\in \{8,10,12\}$ and $t\in \{3,4,6\}$ then by Remark \ref{remark maximal},  we obtain $o(G)\leq s+t$ as the identity element belongs to every maximal cyclic subgroup of $G$. By Lagrange's theorem, $2s\leq o(G)$. It implies that $s\leq t$; a contradiction. 

 \textbf{Subcase-1.7:} $\phi(m_1)=4$ and $\phi(m_2)=\phi(m_3)=2$. Then $m_1\in \{5,8,10,12\}$ and  $m_2, m_3\in \{3,4,6\}$. If $m_1=5$ and $m_2=m_3=3$, then by Remark \ref{remark maximal}, $o(G)=9$. The existence of a subgroup of order $5$ implies that no such group exists. Now assume that at least one of $m_2$ and $m_3$ is not equal to $3$. Let $x$ be an element of order $2$ in $M_2\cup M_3$. Then the subgraph induced by the set $\g \cup \{x\}$ contains a subgraph isomorphic to $K_{4,5}$. Then by Theorem \ref{planarcondition}, we get $\gamma (\d) \geq 2$ and $\overline{\gamma}(\d) \geq 3$.\\
If $m_1\in \{8, 12\}$ and $m_2,\ m_3 \in \{3,4,6\}$ then the subgraph induced by the set $\g \cup \{x,y\}$ contains a subgraph isomorphic to $K_{6,4}$, where $x$ and $y$ are elements of order $4$ in $M_1$. Then by Theorem \ref{planarcondition}, we obtain $\gamma (\d) \geq 2$ and $\overline{\gamma}(\d) \geq 4$. If $m_1=10$ and $m_2,\ m_3 \in \{3,4,6\}$ then the subgraph induced by the set $\g \cup \{x,y,z,t\}$, where $x,y,z$ and $t$ are elements of order $5$ in $M_1$, contains $K_{8,4}$ as a subgraph. It follows that $\gamma (\d) \geq 3$ and $\overline{\gamma}(\d) \geq 6$.

 \textbf{Subcase-1.8:} $\phi(m_1)\geq 4$, $\phi(m_2)\geq 4$ and $\phi(m_3)\ge 1$. Then the subgraph induced by the set $\g$ contains a subgraph isomorphic to $K_{4,5}$. By Theorem \ref{planarcondition}, we get $\gamma (\d) \geq 2$ and $\overline{\gamma}(\d) \geq 3$.
 
 
 \textbf{Subcase-1.9:} $\phi(m_1)\geq 6$, $ \phi(m_2)=2$ and $\phi(m_3)= 1$. Then $m_2\in \{3,4,6\}$ and $m_3=2$. By Remark \ref{remark maximal}, $o(G)\leq m_1+6$. The existence of a subgroup of order $m_1$ follows that $m_1\leq 6$; which is not possible.
 
 \textbf{Subcase-1.10:} $\phi(m_1)\geq 6$, $ \phi(m_2)= 2$ and $\phi(m_3)= 2$. Then the subgraph induced by the set $\g$ contains a subgraph isomorphic to $K_{4,5}$. Then by Theorem \ref{planarcondition}, we get $\gamma (\d) \geq 2$ and $\overline{\gamma}(\d) \geq 3$.
 
 
\noindent\textbf{Case-2: } $|\m |=4$. Let $M_1, M_2, M_3$ and $M_4$ be the maximal cyclic subgroups of $G$ such that $|M_i|=m_i$ for $i\in \{1,2,3,4\}$. Without loss of generality, assume that $\phi(m_1) \geq \phi(m_2) \geq \phi(m_3) \geq \phi(m_4)$. Now, we have the following subcases:

  \textbf{Subcase-2.1:} $\phi(m_1)=\phi(m_2)=\phi(m_3)=\phi(m_4)=1.$ By Subcase-$2.1$ of the Theorem \ref{unicyclic}, no such group exists.
  
  \textbf{Subcase-2.2:} $\phi(m_1)=2$ and $\phi(m_2)=\phi(m_3)=\phi(m_4)=1.$ Then $m_1\in\{3,4,6\}$ and $m_i=2$ for every $i\in \{2,3,4\}$. By Remark \ref{remark maximal}, we obtain $o(G)\in \{6,7,9\}$. Since there exists a subgroup of order $2$, we have $o(G)=6$. By Table \ref{table}, we obtain $G\cong S_3$.
  
\textbf{Subcase-2.3:} $\phi(m_1)=\phi(m_2)=2$ and $\phi(m_3)=\phi(m_4)=1$. Then $m_1,m_2\in \{3,4,6\}$ and $m_3=m_4=2$. If $m_1=3$ and $m_2=t$, where $t\in \{3,4,6\}$, then by Remark \ref{remark maximal}, we obtain $o(G)\in \{7,8,10\}$. The existence of a subgroup of order $3$ implies that no such group exists.  If $m_1=4$ and $m_2=t$, where $t\in \{3,4,6\}$, then by Remark \ref{remark maximal}, we get $o(G)\in \{8,9,10,11\}$. Since there exists a subgroup of order $4$, we have $o(G)=8$. By Table \ref{table}, we obtain $G\cong \mathbb{Z}_2\times \mathbb{Z}_4$.
    If $m_1=6$ and $m_2=t$, where $t\in \{3,4,6\}$, then  $o(G)\in \{10,11,12,13\}$ (cf. Remark \ref{remark maximal}). The existence of a subgroup of order $6$ gives $o(G)=12$. By Table \ref{table}, no such group exists.
    
   \textbf{Subcase-2.4:} $\phi(m_1)=\phi(m_2)=\phi(m_3)=2$ and $\phi(m_4)=1.$ Then $m_1,m_2,m_3\in  \{3,4,6\}$ and $m_4=2$. For distinct  $i,j,k \in \{1,2,3\}$, if $m_i=m_j=3$ and $m_k=t$, where $t\in \{3,4,6\}$, then by Remark \ref{remark maximal}, $o(G)\in \{8,9,11\}$. The existence of subgroups of order $2$ and $3$, follows that that no such group exists.\\
    For distinct  $i,j,k \in \{1,2,3\}$, if $m_i=3$ and $m_j=m_k=4$, then $|M_j\cap M_k|\in \{1,2\}$. Consequently, we get $o(G)\in \{9,10\}$ (see Remark \ref{remark maximal}), which is not possible because $G$ has a subgroup of order $4$. For distinct  $i,j,k \in \{1,2,3\}$, if $m_i=m_j=4$ and $m_k=6$,  then notice that the cardinality  of  intersection of any two of these maximal cyclic subgroups is at most $2$. Consequently, by Remark \ref{remark maximal}, $o(G)\in \{11,12,13\}$. Since $G$ has a subgroup of order $2$, we have $o(G)=12$. By Table \ref{table}, no such group exists. If $m_1=m_2=m_3=4$ then by Remark \ref{remark maximal}, $o(G)\in \{9,10,11\}$. The existence of a subgroup of order $4$ implies that no such group exists.\\
   For distinct  $i,j,k \in \{1,2,3\}$,  if $m_i=m_j=6$ and $m_k=3$, then $|M_i\cap M_j|\in \{1,2,3\}$. Consequently, by Remark \ref{remark maximal}, $o(G)\in \{12,13,14\}$. The existence of a subgroup of order $3$ implies that $o(G)=12$. By Table \ref{table}, no such group exists. For distinct  $i,j,k \in \{1,2,3\}$, if $m_i=m_j=6$ and $m_k=4$, then by Remark \ref{remark maximal}, $o(G)\in \{12,13,14,15\}$ and therefore $o(G)=12$. By Table \ref{table}, no such group exists.
   If $m_1=m_2=m_3=6$, then by Remark \ref{remark maximal}, $13 \leq o(G)\leq 17$. Consequently, the existence of a subgroup of order $6$ implies that $o(G)=12$. By Table \ref{table}, no such group exists. For distinct  $i,j,k \in \{1,2,3\}$, if $m_i=3,m_j=4$ and $m_k=6$, then $lcm(3,4,6)$ divides $o(G)$. By Remark \ref{remark maximal}, $o(G)\leq 12$. Thus, we obtain $o(G)=12$. By Table \ref{table}, no such group exists. 
   
   \textbf{Subcase-2.5:} $\phi(m_1)=\phi(m_2)=\phi(m_3)=\phi(m_4)=2.$ If $m_i=3$ for every $i\in \{1,2,3,4\}$, then $o(G)=9$ and therefore by Table \ref{table}, $G\cong \mathbb{Z}_3\times \mathbb{Z}_3$. 

   Notice that if $x\notin \g$ is an element of $M_i$, for $i\in \{1,2,3,4\}$, such that $x$ belongs to at most two maximal cyclic subgroups then the subgraph induced by $\g \cup \{x\}$ contains a subgraph isomorphic to $K_{4,5}$. For distinct $i,j,k,l\in \{1,2,3,4\},$ let $m_i=6$, $m_j,m_k,m_l\in \{3,4,6\}$ and let $a,b \in M_i$ such that $o(a) = 2$, $o(b) = 3$. Then either $a$ or $b$ belongs to at most two maximal cyclic subgroups. Consequently, $G$ contains a subgraph isomorphic to $K_{4,5}$. Then by Theorem \ref{planarcondition}, we get $\gamma (\d) \geq 2$ and $\overline{\gamma}(\d) \geq 3$. For distinct $i,j,k,l\in \{1,2,3,4\}$, if $m_i=3$ and $m_j=m_k=m_l=4$,  then by Remark \ref{remark maximal}, we get $o(G)\in \{10,11,12\}$. The existence of a subgroup of order $3$ follows that $o(G)=12$. By Table \ref{table}, no such group exists.  If $m_i=4$ for every $i\in \{1,2,3,4\}$, then by Remark \ref{remark maximal}, we get $o(G)\in \{10,11,12,13\}$. The existence of a subgroup of order $4$ implies that $o(G)=12$. By Table \ref{table}, no such group exists.
    
    \textbf{Subcase-2.6:} $\phi(m_1)\geq 4$ and $\phi(m_2)=\phi(m_3)=\phi(m_4)=1.$ By the similar argument used in Subcase-$1.5$ of the Theorem \ref{unicyclic}, no such group exists.
    
     \textbf{Subcase-2.7:} $\phi(m_1)= 4$,\ $\phi(m_2)=2$ and $\phi(m_3)=\phi(m_4)=1$. Then $m_1\in \{5,8,10,12\}$, $m_2\in \{3,4,6\}$ and $m_3=m_4=2$.  If $m_1=5$, $m_2=t$, where $t\in \{3,4,6\}$, then $o(G)\in \{9,10,12\}$ (see Remark \ref{remark maximal}). Since there exists a subgroup of order $5$, we get $o(G)=10$. By Table \ref{table}, no such group exists. If $m_1=s$, $m_2=t$, where $s\in \{8,10,12\}$, and $ t\in \{3,4,6\}$, then $M_1$ has an element $x$ such that $\langle x \rangle \neq M_1$ and $x\notin M_2\cup M_3 \cup M_4$. Then the graph induced by the set $\g \cup \{x\}$ contains a subgraph isomorphic to $K_{4,5}$. Then by Theorem \ref{planarcondition}, we obtain $\gamma (\d) \geq 2$ and $\overline{\gamma}(\d) \geq 3$.
    
     \textbf{Subcase-2.8:} $\phi(m_1)\geq 4$, \ $\phi(m_2)\geq 2$, \  $\phi(m_3)\geq 2 $ and $\phi(m_4)\geq 1$. Then the graph induced by the set $\g $ contains a subgraph isomorphic to $K_{4,5}$. Then by Theorem \ref{planarcondition}, we obtain $\gamma (\d) \geq 2$ and $\overline{\gamma}(\d) \geq 3$.
     
      \textbf{Subcase-2.9:} $\phi(m_1)\geq 4$, \ $\phi(m_2)\geq 4$, $\phi(m_3)= 1 $ and $\phi(m_4)= 1.$ Then the graph induced by the set $\g $ contains a subgraph isomorphic to $K_{6,5}$. Then by Theorem \ref{planarcondition}, we obtain $\gamma (\d) \geq 2$ and $\overline{\gamma}(\d) \geq 4$.
     
    \textbf{Subcase-2.10:} $\phi(m_1)\geq 6$, \ $\phi(m_2)=2$ and $\phi(m_3)=\phi(m_4)=1.$ Then the subgraph induced by the set $\g$ contains a subgraph isomorphic to $K_{4,6}$. Consequently, by Theorem \ref{planarcondition}, we obtain $\gamma (\d) \geq 2$ and $\overline{\gamma}(\d) \geq 4$.

 \noindent\textbf{Case-3: } $|\m |=5$. Let $M_1, M_2, M_3, M_4$ and $M_5$ be the maximal cyclic subgroups of $G$ such that $|M_i|=m_i$ for $i\in \{1,2,3,4,5\}$. Without loss of generality, assume that $\phi(m_1) \geq \phi(m_2) \geq \phi(m_3) \geq \phi(m_4)\geq \phi(m_5)$. Now, we have the following subcases: 
 
  \textbf{Subcase-3.1:} $\phi(m_1)=\phi(m_2)=\phi(m_3)=\phi(m_4)=\phi(m_5)=1.$ Then $m_i=2$ for every $i\in \{1,2,3,4,5\}$. Consequently by Remark \ref{remark maximal}, $o(G)=6$. By Table \ref{table}, no such groups exists.
  
  \textbf{Subcase-3.2:} $\phi(m_1)=2$ and $\phi(m_2)=\phi(m_3)=\phi(m_4)=\phi(m_5)=1.$ It follows that $m_1\in \{3,4,6\}$ and $m_i=2$ for every $i\in \{2,3,4,5\}$. Consequently, by Remark \ref{remark maximal}, $o(G)\in \{7,8,10\}$. Note that $o(G)=7$ if $m_1=3$, but $3$ does not divide $7$ and $o(G)=10$ if $m_1=6$ but $6$ does not divide $10$. It follows that $o(G)=8$. By Table \ref{table}, $G\cong D_8$.
  
   \textbf{Subcase-3.3:} $\phi(m_1)=\phi(m_2)=2$ and $\phi(m_3)=\phi(m_4)=\phi(m_5)=1.$ Then $m_1,m_2\in \{3,4,6\}$ and $m_3=m_4=m_5=2$. For distinct $i,j\in \{1,2\}$, if $m_i=3$ and $m_j=t$, where $t\in \{3,4,6\}$, then $o(G)\in \{8,9,11\}$. Since there exists  subgroups of order $2$ and $3$, it follows that no such group exists. For distinct $i,j\in \{1,2\}$, if $m_i=4$ and $m_j=t$, where $t\in \{4,6\}$, then $o(G)\in \{9,10,11,12\}$ (see Remark \ref{remark maximal}). The existence of a subgroup of order $4$ implies that $o(G)=12$. By Table \ref{table}, no such groups exists. If $m_1=m_2=6$, then by Remark \ref{remark maximal}, $o(G)\in \{12,13,14\}$ and therefore $o(G)=12$. By Table \ref{table}, no such groups exists.
  
  \textbf{Subcase-3.4:} $\phi(m_1)=\phi(m_2)=\phi(m_3)=2$ and $\phi(m_4)=\phi(m_5)=1$. It follows that $m_1,m_2,m_3\in \{3,4,6\}$ and $m_4=m_5=2$. If $m_1=m_2=m_3=3$, then by Remark \ref{remark maximal}, $o(G)=9$. But there exists a maximal cyclic subgroup of order $2$. Thus, no such group exists. For distinct $i,j,k\in \{1,2,3\}$, let $m_i=3$, $m_j=s$ and $m_k=t$, where $s\in \{3,4,6\}$ and $t\in \{4,6\}$. Further, suppose that $x$ is the element of order $2$ in $M_k$. Then the graph induced by the set $\g \cup \{x\}$ contains $K_{4,5}$ as a subgraph. Consequently, by Theorem \ref{planarcondition}, we obtain $\gamma (\d) \geq 2$ and $\overline{\gamma}(\d) \geq 3$. For distinct $i,j,k\in \{1,2,3\}$, let $m_i=4$, $m_j=6$ and $m_k=t$, where $t\in \{4,6\}$. Further, suppose that $x$ is an element of order $3$ in $M_j$. Then the subgraph of $\d$ induced by the set $\g \cup \{x\}$ contains a subgraph isomorphic to $K_{4,5}$. Consequently, by Theorem \ref{planarcondition}, we obtain $\gamma (\d) \geq 2$ and $\overline{\gamma}(\d) \geq 3$.\\
   If $m_1=m_2=m_3=4$ then $o(G)\in \{10,11,12\}$ (see Remark \ref{remark maximal}). The existence of a subgroup of order $4$ implies that $o(G)=12$. By Table \ref{table}, no such group exists. If $m_1=m_2=m_3=6$, then the subgraph induced by the set $\g \cup \{x,y\}$ contains $K_{4,5}$ as a subgraph, where $x$ and $y$ are elements of order $2$ and $3$, respectively, in $M_1$. 
   
  \textbf{Subcase-3.5:} $\phi(m_1)\geq 2, \ \phi(m_2)\geq 2, \ \phi(m_3)\geq 2, \ \phi(m_4)\geq 2$ and $\phi(m_5)\geq 1$. It follows that the subgraph of $\d$ induced by the set $\g$ contains a subgraph isomorphic to $K_{4,5}$. Consequently, by Theorem \ref{planarcondition}, we obtain $\gamma (\d) \geq 2$ and $\overline{\gamma}(\d) \geq 3$.
  
    \textbf{Subcase-3.6:} $\phi(m_1)\geq 4$ and $\phi(m_i)=1$ for every $i\in \{2,3,4,5\}$. By using the similar argument given in Subcase-$1.5$ of Theorem \ref{unicyclic}, no such group exists.
    
    \textbf{Subcase-3.7:} $\phi(m_1)\geq 4, \ \phi(m_2)\geq 2$ and $\phi(m_i)\geq 1$ for every $i\in \{3,4,5\}$. By the similar argument used in Subcase-$3.5$, we get $\gamma (\d) \geq 2$ and $\overline{\gamma}(\d) \geq 3$.
    
\noindent\textbf{Case-4: } $|\m |=6$. Let $M_1, M_2, M_3, M_4, M_5$ and $M_6$ be the maximal cyclic subgroups of $G$ such that $|M_i|=m_i$ for $i\in \{1,2,3,4,5,6\}$. Without loss of generality, assume that $\phi(m_1) \geq \phi(m_2) \geq \phi(m_3) \geq \phi(m_4)\geq \phi(m_5)\geq \phi(m_6)$. Now, we have the following subcases:

  \textbf{Subcase-4.1:} $\phi(m_i)=1$ for each $i\in \{1,2,3,4,5,6\}.$ It follows that $o(G)=7$ which is not possible because $|M_i|=2$.
  
   \textbf{Subcase-4.2:} $\phi(m_1)=2, \ \phi(m_i)=1$ for each $i\in \{2,3,4,5,6\}$. Then $m_1\in \{3,4,6\}$. Consequently, by Remark \ref{remark maximal}, $o(G)\in \{8,9,11\}$. Note that $o(G)=8$ if $m_1=3$ but $3$ does not divide $8$, $o(G)=9$ if $m_1=4$ but $4$ does not divide $9$ and $o(G)=11$ if $m_1=6$ but $6$ does not divide $11$. Thus, no such group exists.
   
    \textbf{Subcase-4.3:} $\phi(m_1)=\phi(m_2)=2$ and $\phi(m_i)=1$ for every $i\in \{3,4,5,6\}$. It follows that $m_1,m_2\in \{3,4,6\}$ and $m_i=2$. If $m_1=m_2=3$ then by Remark \ref{remark maximal}, $o(G)=9$. No such groups exists because we have a maximal cyclic subgroup of order $2$. For distinct $i,j\in \{1,2\}$, if $m_i\geq 3$ and $m_j=t$, where $t\in \{4,6\}$, then the subgraph of $\d$ induced by the set $\g \cup \{x\}$, where $x$ is the elements of order $2$ in $M_j$, contains a subgraph isomorphic to $K_{4,5}$. Consequently, we get $\gamma (\d) \geq 2$ and $\overline{\gamma}(\d) \geq 3$.
   
     \textbf{Subcase-4.4:} $\phi(m_1)\geq 2, \ \phi(m_2)\geq 2, \ \phi(m_3)\geq 2 \ $ and $\phi(m_i) \ge 1$ for every $i\in \{4,5,6\}.$ By the similar argument given in Subcase-$3.5$, $\gamma (\d) \geq 2$ and $\overline{\gamma}(\d) \geq 3$.
     
   \textbf{Subcase-4.5:} $\phi(m_1)\geq 4, \ \phi(m_i)\geq 1$ for every $i\in \{2,3,4,5,6\}.$ By the similar argument given in Subcase-$3.5$, $\gamma (\d) \geq 2$ and $\overline{\gamma}(\d) \geq 3$.
   
      \noindent\textbf{Case-5: } $|\m |=7$. Now, we have the following subcases:
      
 \textbf{Subcase-5.1:} $\phi(m_i)=1$ for every $i\in \{1,2,3,4,5,6,7\}.$ Then by Remark \ref{remark maximal}, $o(G)=8$. By Table \ref{table}, $G\cong \mathbb{Z}_2\times \mathbb{Z}_2\times \mathbb{Z}_2$.
 
   \textbf{Subcase-5.2:} $\phi(m_1)=2, \ \phi(m_i)=1$ for every $i\in \{2,3,4,5,6,7\}.$ It follows that  $m_1\in \{3,4,6\}$ and $m_i=2$. Consequently, by Remark \ref{remark maximal}, $o(G)\in \{9,10,12\}$. Note that $o(G)=9$ when $m_1=3$. But the existence of subgroup of order $2$ makes $o(G)=9$ impossible. Also, $o(G)=10$ when $m_1=4$. But this is not possible because of the existence of a subgroup of order $4$. If $o(G)=12$, then by Table \ref{table}, we get $G\cong D_{12}$.
   
    \textbf{Subcase-5.3:} $\phi(m_1)\geq 2, \ \phi(m_2)\geq 2,$ and $\phi(m_i)\geq 1$ for every $i\in \{3,4,5,6,7\}.$ By the similar argument used in Subcase-$3.5$, $\gamma (\d) \geq 2$ and $\overline{\gamma}(\d) \geq 3$.
    
   \textbf{Subcase-5.4:} $\phi(m_1)\geq 4, \ \phi(m_i)\geq 1$ for every $i\in \{2,3,4,5,6,7\}.$ Similar to the Subcase-$3.5$, we obtain $\gamma (\d) \geq 2$ and $\overline{\gamma}(\d) \geq 3$.
   
           \noindent\textbf{Case-6: } $|\m |\geq 8$. In this case $\d$ contains a subgraph isomorphic to $K_{8}$. Thus, $\gamma (\d) \geq 2$ and $\overline{\gamma}(\d) \geq 4$.
   
Conversely, If $G\cong \mathbb{Z}_2\times \mathbb{Z}_2$ then $\d \cong K_3$. Thus, $\d$ is an outerplaner graph.\\
If $G\cong S_3$ then $\d$ contains a subgraph isomorphic to $K_4$. Consequently, $\d$ is not an outerplanar graph. The graph $\overline{\mathcal{P}_E(S_3^*)}$ is planar (see Figure \ref{S3}).
\begin{figure}[h]
    \centering
    \includegraphics[width=0.35\textwidth]{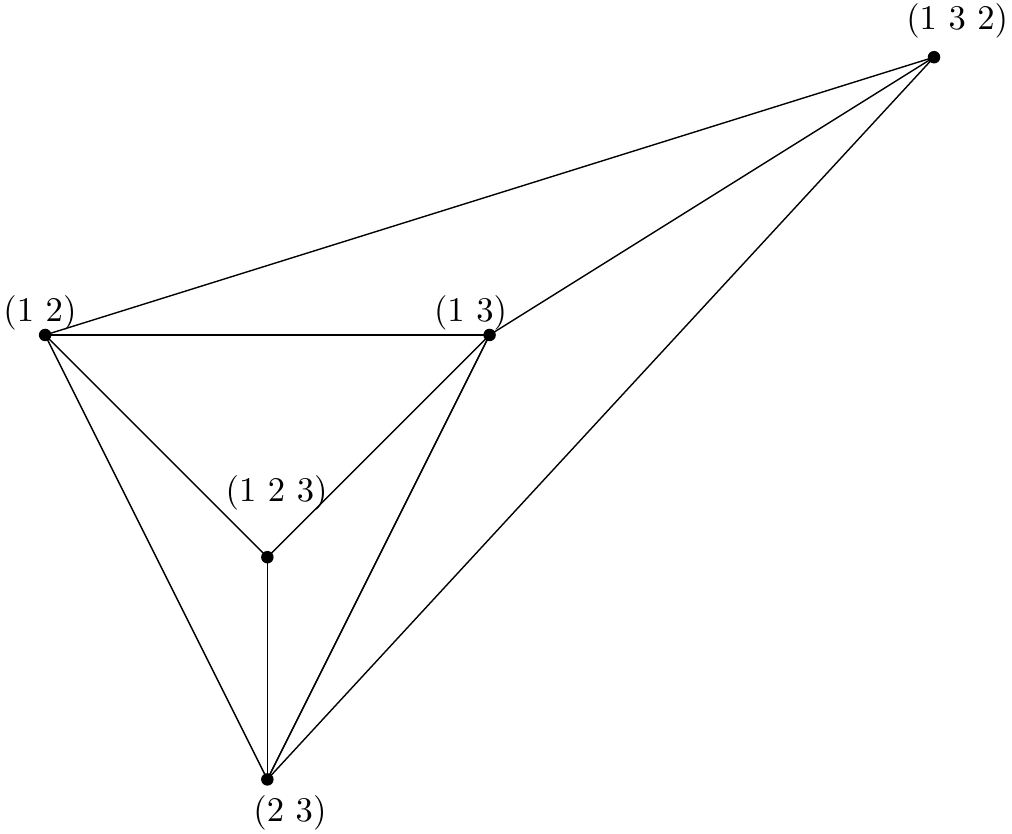}
    \caption{A planar drawing of $\overline{\mathcal{P}_E(S_3^*)}$.}
    \label{S3}
\end{figure}

If $G\cong Q_8$ then note that the subgraph induced by the set $\{i,-i,j,-j,k\}$ contains $K_{2,3}$ as a subgraph. It follows that $\d$ is not an outerplanar graph. Moreover, the graph $\overline{\mathcal{P}_E(Q_8^*)}$ is planar (cf. Figure \ref{Q8}).
\begin{figure}[h!]
    \centering
    \includegraphics[width=0.35\textwidth]{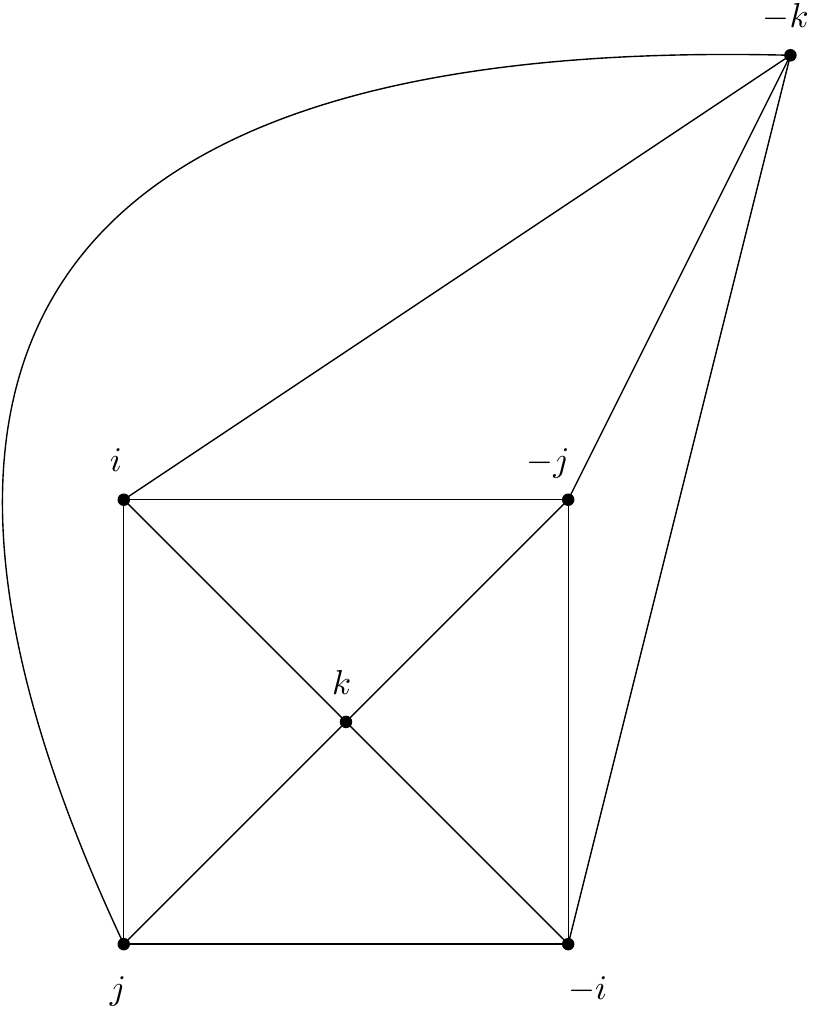}
    \caption{A planar drawing of $\overline{\mathcal{P}_E(Q_8^*)}$.}
    \label{Q8}
\end{figure}

If $G\cong \mathbb{Z}_2\times \mathbb{Z}_4$ then $\d$ contains a subgraph isomorphic to $K_{3,3}$. Consequently, $\d$ is not a planar graph. A toroidal and projective embedding of $\d$ given in the Figure \ref{Z2Z4T} and \ref{Z2Z4}.
\begin{figure}[h]
    \centering
    \includegraphics[width=0.35\textwidth]{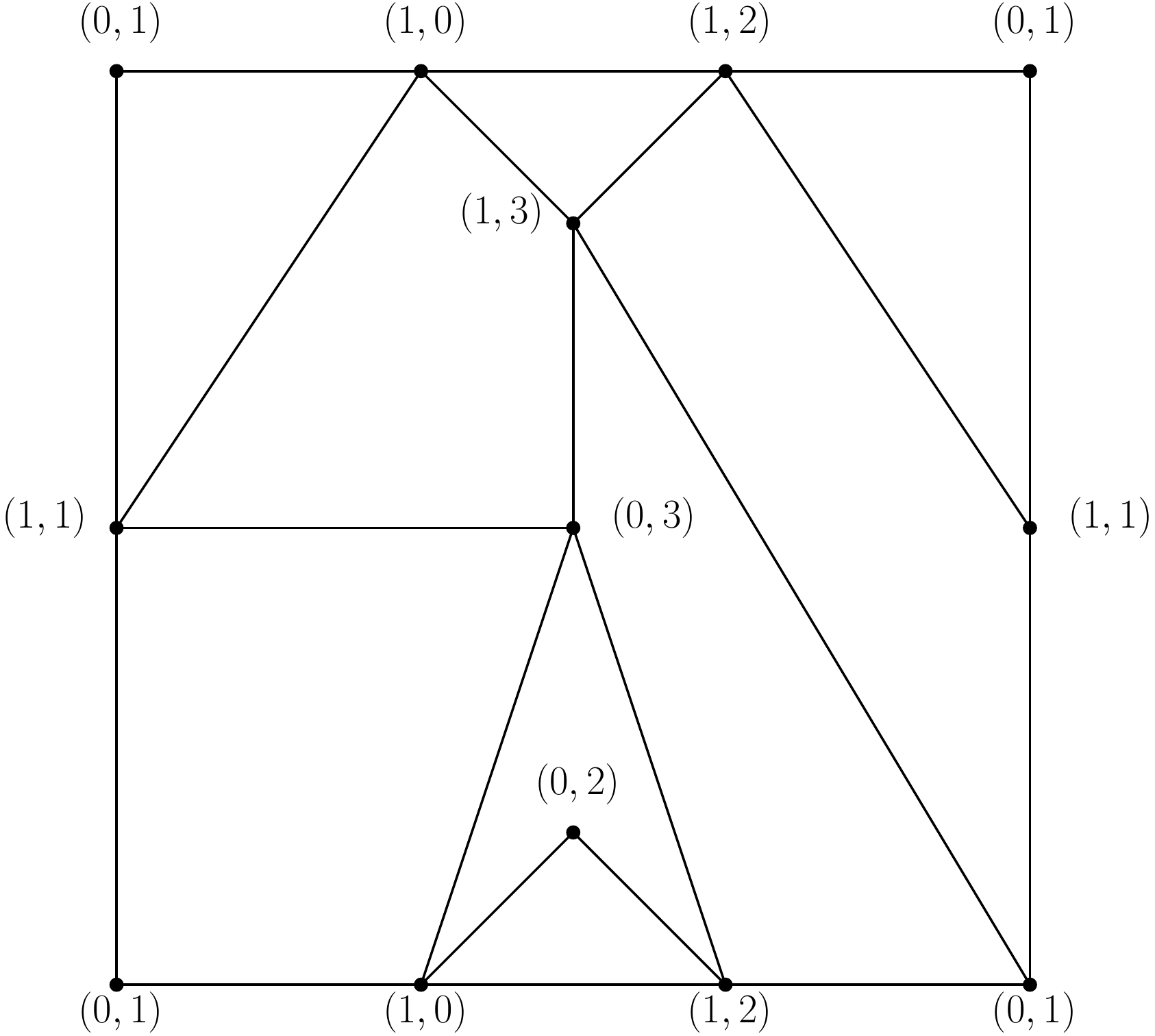}
    \caption{A toroidal embedding of $\overline{\mathcal{P}_E((\mathbb{Z}_2\times \mathbb{Z}_4)^*)}$.}
    \label{Z2Z4T}
\end{figure}\\
\begin{figure}[h]
    \centering
    \includegraphics[width=0.35\textwidth]{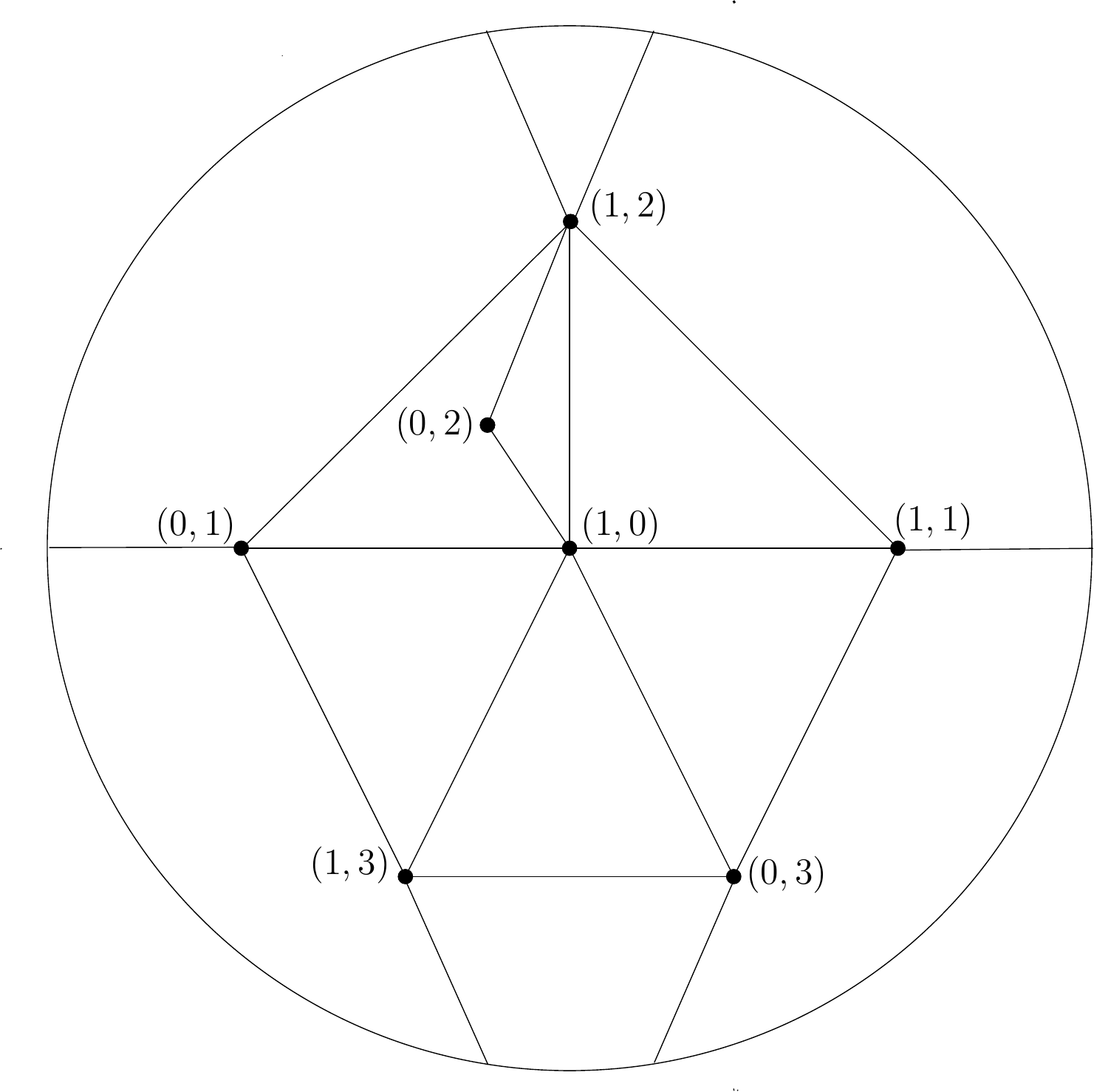}
    \caption{A projective embedding of $\overline{\mathcal{P}_E((\mathbb{Z}_2\times \mathbb{Z}_4)^*)}$.}
    \label{Z2Z4}
\end{figure}

If $G\cong D_8$ then notice that the subgraph of $\d$ induced by the set $\{x,y,xy,x^2y,x^3y\}$ contains $K_5$ as a subgraph. Consequently, $\d$ is not a planar graph. A toroidal and projective embedding of $\d$ given in Figure \ref{D4T} and Figure \ref{D4} respectively.
\begin{figure}[h]
    \centering
    \includegraphics[width=0.35\textwidth]{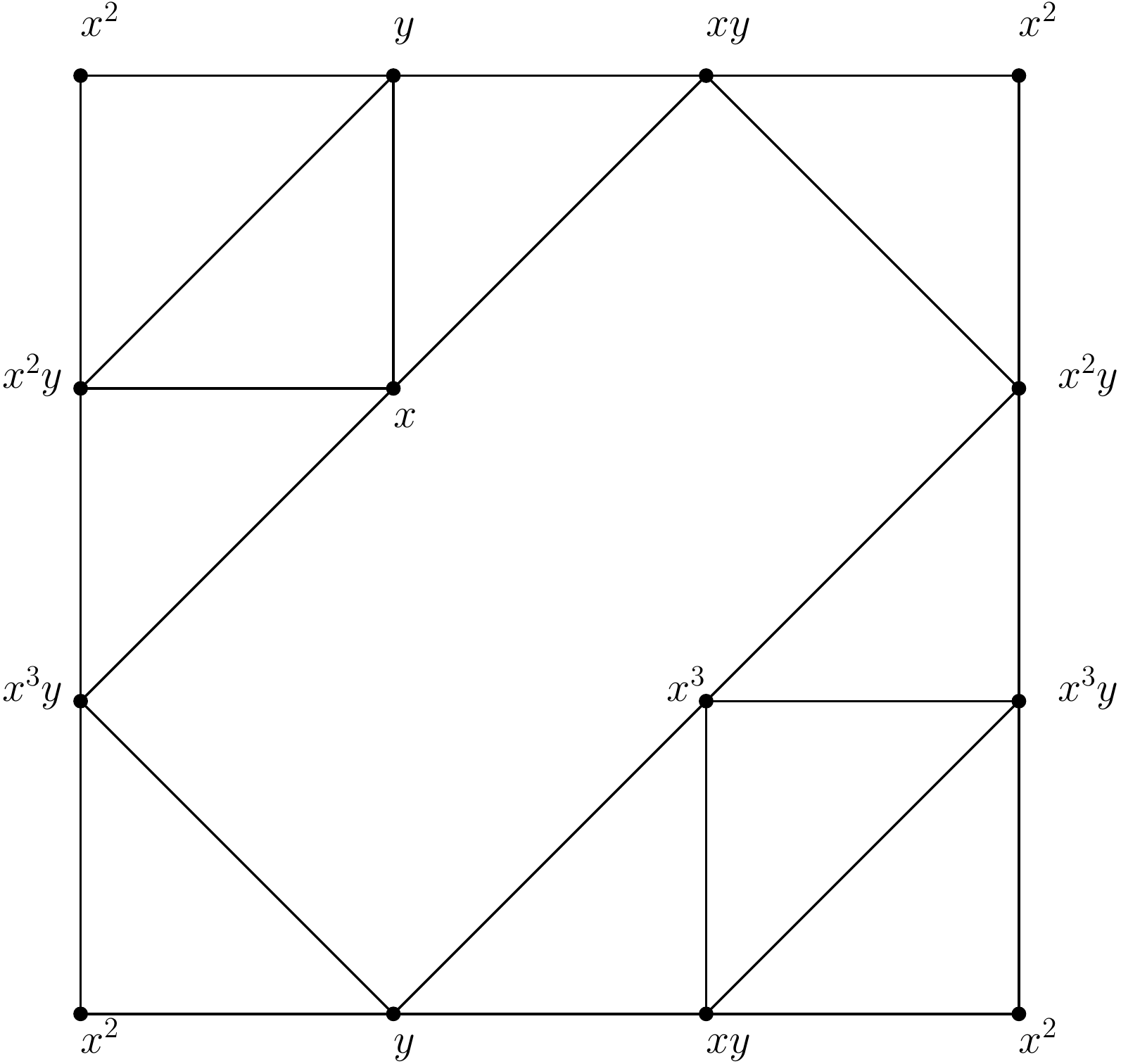}
    \caption{A toroidal embedding of $\overline{\mathcal{P}_E(D_8^*)}$.}
    \label{D4T}
\end{figure}\\
\begin{figure}[h]
    \centering
    \includegraphics[width=0.35\textwidth]{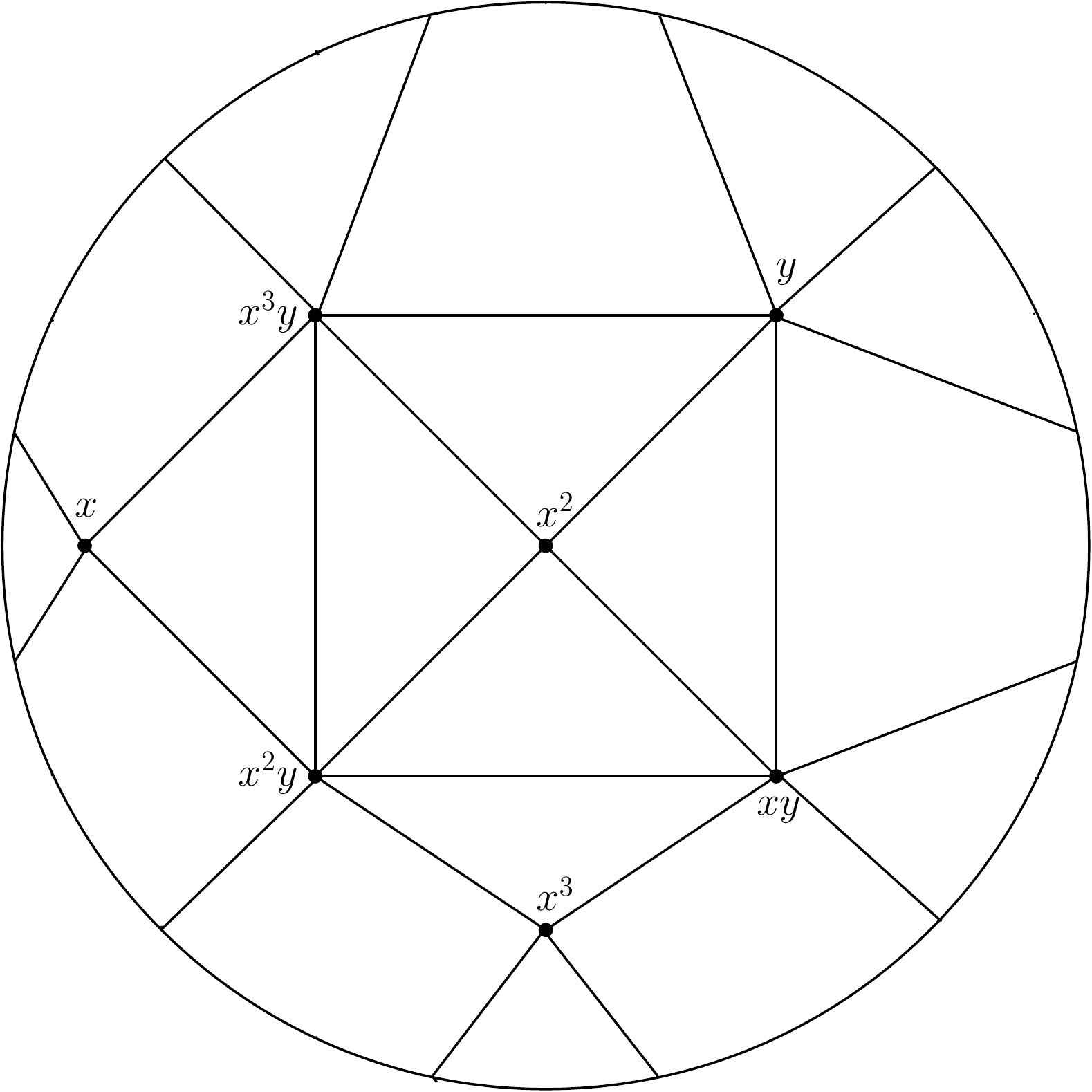}
    \caption{A projective embedding of $\overline{\mathcal{P}_E(D_8^*)}$.}
    \label{D4}
\end{figure}

If $G\cong \mathbb{Z}_3\times \mathbb{Z}_3$ then $\d \cong K_{2,2,2,2}$ and it contains a subgraph isomorphic to $K_{4,4}$. Consequently, $\gamma(\d)\geq 1$ and by \cite{a.jungerman1979nonorientable}, $\overline{\gamma}(\d)= 3$. A toroidal embedding of $\d$ is given in Figure \ref{Z3Z3T}.
\begin{figure}[h]
    \centering
    \includegraphics[width=0.35\textwidth]{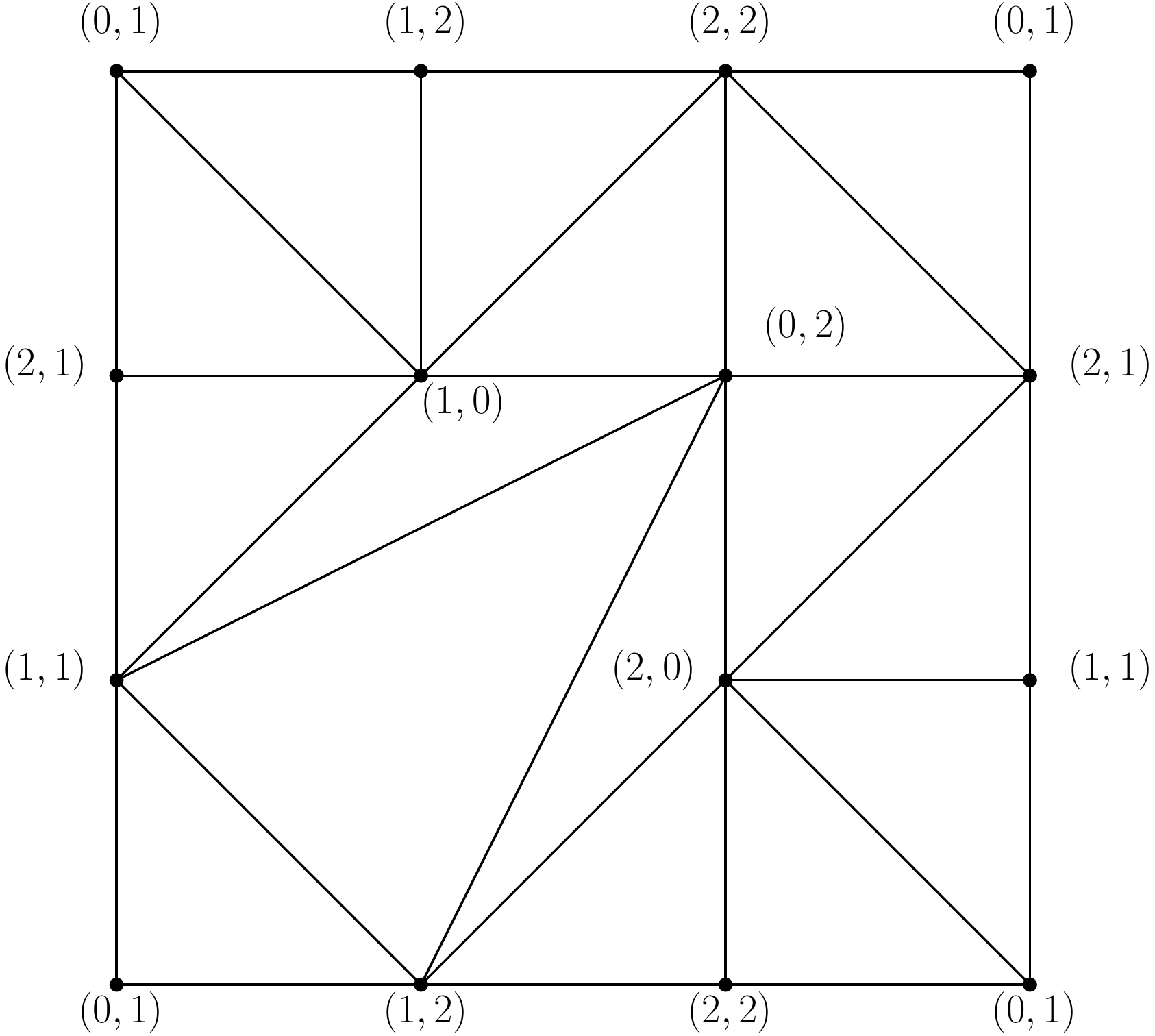}
    \caption{A toroidal embedding of $\overline{\mathcal{P}_E((\mathbb{Z}_3\times \mathbb{Z}_3)^*)}$.}
    \label{Z3Z3T}
\end{figure}

If $G\cong \mathbb{Z}_2\times \mathbb{Z}_6$ then $\d \cong K_{3,3,3}$ contains a subgraph isomorphic to $K_{4,4}$. Consequently, $\gamma(\d)\geq 1$ (see Theorem \ref{planarcondition}) and by {\rm \cite[Theorem 10]{a.Ellingham2006}},  $\overline{\gamma}(\d)= 3$. A toroidal embedding of $\d$ is given in Figure \ref{Z2Z6T}.
\begin{figure}[h]
    \centering
    \includegraphics[width=0.35\textwidth]{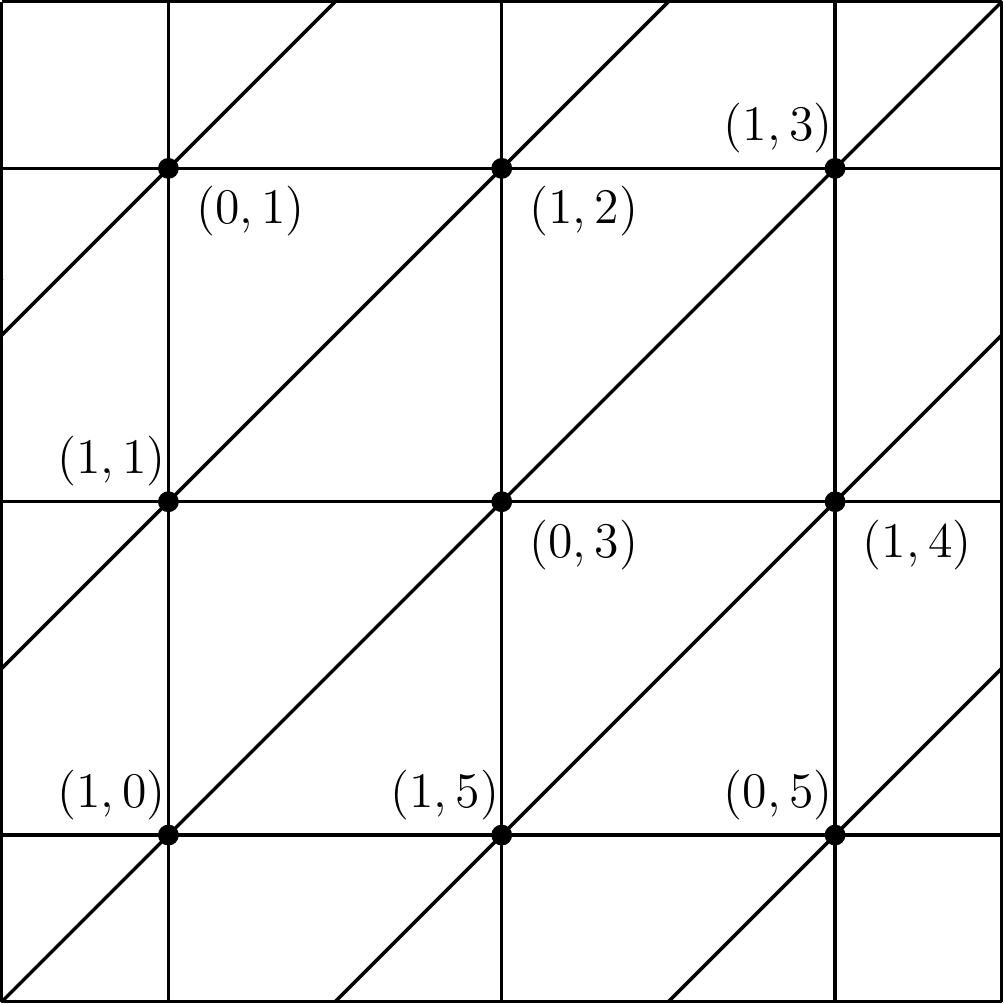}
    \caption{A toroidal embedding of $\overline{\mathcal{P}_E((\mathbb{Z}_2\times \mathbb{Z}_6)^*)}$.}
    \label{Z2Z6T}
\end{figure}\\
If $G\cong \mathbb{Z}_2\times \mathbb{Z}_2\times \mathbb{Z}_2$ then $\d\cong K_7$. Consequently, by Theorem \ref{planarcondition}, $\overline{\gamma}(\d)=3$ and $\gamma(\d)=1$.
If $G\cong D_{12}$ then $\overline{\mathcal{P}_E(G^*)}$ contains a subgraph isomorphic to $K_{5,6}$. Consequently, by Theorem \ref{planarcondition}, $\gamma (\d) \geq 3$ and $\overline{\gamma}(\d) \geq 6$.
\end{proof}
\section{Acknowledgement}

The first author gratefully acknowledge for providing financial support to CSIR  (09/719(0110)/2019-EMR-I) government of India.

\bibliographystyle{abbrv}
\bibliography{References}
\vspace{1cm}
\noindent
{\bf Parveen\textsuperscript{\normalfont 1}, Jitender Kumar\textsuperscript{\normalfont 1}}
\bigskip

\noindent{\bf Addresses}:

\vspace{5pt}

\noindent
\textsuperscript{\normalfont 1}Department of Mathematics, Birla Institute of Technology and Science Pilani, Pilani-333031, India.

\end{document}